\documentclass{article}
\usepackage{amsfonts, amsmath, amsthm, hyperref, bbm}
\usepackage{mathrsfs}
\usepackage[svgnames,x11names]{xcolor}
\usepackage{comment}
\usepackage{enumerate}

\usepackage[margin=3cm]{geometry}
\usepackage[algo2e,ruled,vlined]{algorithm2e}
\usepackage[section]{algorithm}
\usepackage[normalem]{ulem}
\usepackage{algpseudocode}

\newtheorem{theorem}{Theorem}[section]
\newtheorem{corollary}[theorem]{Corollary}
\newtheorem{proposition}[theorem]{Proposition}
\newtheorem{lemma}[theorem]{Lemma}

\newtheorem{remark}[theorem]{Remark}
\newtheorem{example}[theorem]{Example}

\DeclareMathOperator{\jdup}{jdup}
\DeclareMathOperator{\diag}{diag}

\newcommand{\bR}{\mathbb{R}}
\newcommand{\bN}{\mathbb{N}}

\title{Bordering of Symmetric Matrices and an Application to the Minimum Number of Distinct Eigenvalues for the Join of Graphs
}
\author{
Aida Abiad\thanks{Department of Mathematics and Computer Science, Eindhoven University of Technology, Eindhoven, The Netherlands (\texttt{a.abiad.monge@tue.nl}). Department of Mathematics: Analysis, Logic and Discrete Mathematics, Ghent University, Ghent, Belgium. Department of Mathematics and Data Science, Vrije Universiteit Brussel, Brussels, Belgium. Partially supported by the Research Foundation Flanders (FWO) grant 1285921N.}
\and
Shaun M.~Fallat\thanks{Department of Mathematics and Statistics, University of Regina, Regina, SK, S4S 0A2, Canada (\texttt{shaun.fallat@uregina.ca}). Research supported in part by an NSERC Discovery Grant RGPIN--2019--03934.} 
\and
Mark~Kempton\thanks{Department of Mathematics, Brigham Young University, Provo UT 84602, U.S.A. (\texttt{mkempton@mathematics.byu.edu}).}
\and
Rupert H.~Levene\thanks{School of Mathematics and Statistics, University College Dublin, Belfield, Dublin 4, Ireland (\texttt{rupert.levene@ucd.ie} and \texttt{helena.smigoc@ucd.ie}).}
\and
Polona Oblak\thanks{University of Ljubljana, Faculty of Computer and Information Science and Faculty of Mathematics and Physics, Ljubljana, Slovenia. Institute of Mathematics, Physics, and Mechanics, Ljubljana, Slovenia (\texttt{polona.oblak@fri.uni-lj.si}). Partially supported by Slovenian Research Agency (research core funding no.~P1-0222 and project no.~J1-3004).}
\and
Helena \v Smigoc\footnotemark[4]
\and
Michael Tait\thanks{Department of Mathematics \& Statistics, Villanova University, Villanova PA 19085, U.S.A. (\texttt{michael.tait@villanova.edu}). Partially supported by National Science Foundation grant DMS-2011553 and a Villanova University Summer Grant.}
\and
Kevin Vander Meulen\thanks{Department of Mathematics, Redeemer University, ON, L9K 1J4, Canada (\texttt{kvanderm@redeemer.ca}).
Research supported in part by an NSERC Discovery Grant RGPIN--2022--05137.}
}

\date{}

\begin{document}

\maketitle
 
\begin{abstract}
   An important facet of the inverse eigenvalue problem for graphs is to determine the minimum number of distinct eigenvalues of a particular graph. We resolve this question for the join of a connected graph with a path. We then focus on bordering a matrix and attempt to control the change in the number of distinct eigenvalues induced by this operation. By applying bordering techniques to the join of graphs,  we obtain numerous results on the nature of the minimum number of distinct eigenvalues as vertices are joined to a fixed graph.
\end{abstract}



\section{Introduction}

Given a simple graph $G$ on $|G|=n$ vertices, let $S(G)$ denote the set of all $n \times n$ real symmetric matrices $A=\big(a_{ij}\big)$ such that, for $i \neq j$, $a_{ij} \neq 0$ if and only if $i$ and $j$ are adjacent in~$G$. There are no restrictions on the main diagonal entries of $A$. The inverse eigenvalue problem for~$G$ asks which possible multi-sets of eigenvalues (spectra) occur in the class $S(G)$.
This is a very difficult problem for most graphs (which generally remains open, except for some sporadic graphs, including, for example, paths, cycles, complete graphs and some basic families of trees). Considerable work on this important problem has occurred over the past several decades (see the recent book \cite{HLS2022inverse}). Our work generally pertains to multiplicity lists associated to the spectra of matrices in $S(G)$.

Suppose $A$ is an $n \times n$ real symmetric matrix and $\lambda$ is an eigenvalue of $A$, that is $\lambda \in \sigma(A)$, where $\sigma(A)$ denotes the collection of eigenvalues (spectrum) of the matrix $A$. We let $m_{A}(\lambda)$ denote the multiplicity of $\lambda$ in $\sigma(A)$; if a scalar $\lambda$ is not an eigenvalue of a matrix $A$ then we define $m_A(\lambda)=0$.
Perhaps one of the most important results on the eigenvalues of real symmetric matrices is Cauchy's interlacing inequalities, from which it immediately follows that 
if $A$ is an $n\times n$ principal submatrix 
of an $(n+1) \times (n+1)$ real symmetric matrix~$B$, then 
$|m_A(\lambda)-m_B(\lambda)|\leq 1$ for any scalar $\lambda$. Another way to view the principal submatrix $A$ of $B$ is to consider that $B$ was obtained from $A$ by bordering $A$ with one row and column, and since the spectrum is invariant under permutation similarity, we might as well assume that the new row and column added to $A$ are the first row and column of $B$. More generally, given a symmetric $n\times n$ matrix $A$ and $r\ge 1$, an \emph{$r$-bordering} of $A$ is any symmetric $(n+r)\times (n+r)$ matrix $B$ which contains $A$ as a trailing principal $n\times n$ submatrix (that is, $A$ lies in rows and columns indexed by $\{r+1, r+2, \ldots, n+r\}$ of $B$), and it follows that $|m_A(\lambda)-m_B(\lambda)|\leq r$. For brevity, we will also let $A[S]$ denote the principal submatrix of $A$ lying in rows and columns indexed by $S\subseteq \{1,2,\ldots, n\}$.

 We define the {\em maximum multiplicity} of a symmetric matrix $A$ to be $$M(A)=\max\{m_A(\lambda):\lambda\in \sigma(A)\},$$ and the \emph{maximum multiplicity of a graph~$G$} is
\[ M(G) = \max\{ M(A) : A \in S(G)\}.\]
Let ${\bf m}=(m_1,\ldots,m_k)\in \bN_0^k$ be a sequence of $k$ nonnegative integers
and 
$q({\bf m})=|\{i\colon m_i>0\}|$.
We say ${\bf m}$ is an ordered multiplicity list for a symmetric matrix $A$, if $A$ possesses $q({\bf m})$ distinct eigenvalues $\lambda_1 < \lambda_2 < \cdots < \lambda_{q({\bf m})}$
and $m_A(\lambda_i)=m_{j_i}$ for $i=1,2,\ldots,q({\bf m})$, where $1\le j_1<j_2<\dots<j_{q(\bf m)}\le k$ are the $q({\bf m})$ indices $j$ with $m_j>0$.
In this case we write ${\bf m} = {\bf m}(A)$. For any matrix $A$, we write $q(A)=k$ if $A$ has $k$ distinct eigenvalues. For a given graph $G$, we define
\[ q(G)= \min\{q(A) : A\in S(G)\}.\]
 It is easy to observe that for any graph $G$ we have  $q(G)\ge \lceil\frac{|G|}{M(G)}\rceil$.
 In this paper our goal is to  investigate the behaviour of $q(\cdot)$ upon appending vertices to a fixed graph $G$. Here, when a vertex is appended, all possible edges between the existing vertices and the new vertex are inserted. 

We let $K_n$ ($n\geq 1$), $P_n$ ($n\geq 1$), $C_n$ ($n\geq 3$) denote the complete graph, the path, and the cycle on $n$ vertices.
If $G$ and $H$ are two graphs, then the {\em join of $G$ and $H$}, denoted by $G \vee H$, is the graph obtained from the union of $G$ and $H$ by adding all edges with one endpoint in $G$ and one endpoint in $H$. Hence, our goal in this paper is to investigate the behaviour of $q(G\vee H)$ for various graphs $G$ and $H$.

Given a graph~$G$, let $V(G)$ denote its vertex set. For $v \in V(G)$, we define
$\jdup(G,v)$ to be the supergraph of $G$ obtained from $G$ by duplicating $v$, with an edge connecting $v$ to its duplicate. That is, $V(\jdup(G,v))=V(G)\cup \{w\}$, where $w\not\in V(G)$, and $\{v,w\}\in E(\jdup(G,v))$, and $w$ has the same neighbours as $v$ in $\jdup(G,v)$. As observed in 
\cite[Theorem~3]{MR4284782} and \cite[Lemma~2.9]{MR3891770}, 
\begin{equation}\label{eq:jdup}
q(\jdup(G,v))\leq q(G).
\end{equation}
Since $K_{n+1}\vee H=\jdup(K_n\vee H,v)$ for any vertex $v\in K_n$, we see that $q(K_n\vee H)$ is monotone decreasing in $n$.

One of the first examples considered along these lines was the case of determining $q(K_1 \vee P_n)$. In \cite[Example~4.5]{MR3904092} it was shown that  $q(K_1\vee P_n)=\lceil \frac{n+1}2\rceil$ for $n\geq 2$.  We note here that the lower bound on $q(K_1\vee P_n)$ follows from Cauchy's interlacing inequalities since $q(P_n)=n$. 
Another important example is the star, or $S_{n} =K_1 \vee E_{n-1}$, where $E_{k}$ represents the empty graph on $k$ vertices. It is straightforward to show that $M(S_n)=n-2$ and that $q(S_n)=3$. We remark that the star has played a key role in the inverse eigenvalue problem for graphs (mostly in the case of trees), and in many ways was a critical tool used in \cite{MR978591} to establish a converse to Cauchy's interlacing inequalities. This technique has been extended and adapted to broaden the scope of which spectra can be realized by a graph that contains a dominating vertex (see, for example, \cite{MR2294340, MR2022294, levene2021paths}).


Merging the concepts of bordering a particular matrix and joining a vertex to a given graph, we are interested in determining the minimum number of distinct eigenvalues of a graph joined by a sequence of vertices, and we develop techniques, based in part of the nature of ordered multiplicity lists and eigenvectors, to aid this computation. We begin, in Section 2, with the necessary background and present a general upper bound (Theorem \ref{thm:general_upper_bound}) on $q(G \vee H)$ for connected graphs $G$ and $H$, which reduces to a simple exact formula in the case $G=P_n$. In Section~\ref{sec:border} we investigate the borderings of a given symmetric matrix. Theorem~\ref{thm:1-bordering} describes in detail how a $1$-bordering can change the spectrum of a symmetric matrix, and in Proposition~\ref{prop:C-bordering} we find a necessary and sufficient condition for the existence of an $r$-bordering of a symmetric matrix with a given value of $q$.
In Section~\ref{sec:joinswithcompletegraphs} we make several observations on the patterns of such bordered matrices, and we apply them to estimate $q(K_n\vee H)$ when $H$ is either a hypercube or a cycle. Finally, in Section~\ref{sec:limitations}, we pay particular attention to some possible limitations of our methods (Corollary~\ref{coro:limitationsnotsufficientcondition}).



\section{General graphs and paths}\label{sec:generalgraphsandpaths}

It is known that if $G$ and $H$ are two connected graphs and $|G|=|H|$, then $q(G \vee H)=2$ (see~\cite[Theorem 5.2]{MR3506498}). This result was extended in \cite{levene2021paths,LEVENE2022213} where it was shown that $q(G \vee H)=2$ if $G$ and $H$ are connected graphs with $\big||G|-|H|\big|\leq 2$. Moreover,  for trees $T_1$ and $T_2$ we have
$q(T_1\vee T_2)=2$ if and only if $ \big||T_1|-|T_2|\big|\le2$, so in this case the result is sharp.

An important notion used in \cite{levene2021paths} is generic realizability. Recall that a matrix (vector) is said to be \emph{nowhere zero} if none of its entries is zero. Suppose $G$ is a graph with $|G|=n$ vertices and $\sigma$ is a collection of realizable eigenvalues in $S(G)$ (with multiplicities), i.e., $\sigma=\sigma(A)$ for some $A\in S(G)$. The collection~$\sigma$ is said to be {\em generically realizable} in $S(G)$ if, for any finite set~$\mathcal{Y}$ of nonzero vectors in $\bR^n$, there is an orthogonal matrix $U$ such that $Uy$ is nowhere zero for all $y \in \mathcal{Y}$, and $UDU^{T} \in S(G)$, where $D$ is a diagonal matrix with eigenvalues equal to~$\sigma$ (see~\cite{levene2021paths} for more details). Observe that if we form an $n\times |\cal{Y}|$ matrix $Y$ from the columns of $\mathcal{Y}$ then this property ensures that there is an orthogonal matrix $U$ so that $UDU^T\in S(G)$ and $UY$ is nowhere zero, whenever $Y$ has no zero column (that is, no column of $Y$ is the zero vector in $\bR^n$). We will use the following result.

\begin{theorem}\label{generic}\cite[Theorem~2.5]{levene2021paths}
  Suppose $G$ is a connected graph. Then any $\sigma$ with $|G|$ distinct elements is generically realizable in $S(G)$. 
\end{theorem}

Theorem \ref{generic} allows us to construct matrices in $S(G \vee H)$ with some desired spectral properties, using matrices $A\in S(G)$ and $B \in S(H)$ with distinct eigenvalues. In particular, in the next result we explore this idea of constructing matrices in $S(G \vee H)$ with bounded number of distinct eigenvalues.


\begin{theorem}\label{thm:general_upper_bound}
Suppose $G$ and $H$ are two connected graphs. 
If $k$ is a positive integer
and $|G|\leq |H|\leq k |G|+k+1$, then $$q(G\vee H) \leq k+1.$$ 
In particular, for any connected graphs $G$ and $H$ with $\max\{|G|,|H|\}\ne1$ we have: 
\[ q(G\vee H)\le \left\lceil\frac{|G|+|H|}{\min\{|G|, |H|\}+1}\right\rceil.\]
\end{theorem}
\begin{proof}
Suppose $|G|=n$, $|H|=m$ and $k$ is a positive integer with $n\le m\le kn+k+1$. To prove the first claim, we will construct a matrix in $S(G\vee H)$ with distinct eigenvalues contained in $\mathcal S:=\{\lambda_j\}_{j=1}^{k+1}$ for any chosen set~$\mathcal S$ of $k+1$ distinct numbers. To this end, choose real numbers $\lambda_1 < \cdots < \lambda_{k+1}$, and integers $k_i$ with $1\le k_i\le k$ for $i=1,\ldots,n$, satisfying:
\begin{align*}
0\leq k':=m-\sum_{i=1}^nk_i &\leq k+1.
\end{align*}
Now select $n$ sets of real numbers 
$\mathcal M_i:=\{\mu_{i,1}, \ldots, \mu_{i,{k_i}}\}$, $ i=1,\ldots,n$, where we assume $$\mu_{i,1} <  \cdots < \mu_{i, {k_i}}.$$ Furthermore, we assume that $\mathcal M_i$  strictly interlaces $\{\lambda_1,\ldots,\lambda_{k_i+1}\}$, that the numbers $\mu_{i,j}$ for $j=1,\ldots,k_i$ and $i=1,\ldots,n$ are all distinct, and finally we demand that numbers $a_i: = \left(\sum_{j=1}^{k_i+1} \lambda_j \right) - \left( \sum_{j=1}^{k_i} \mu_{i,j} \right)$, $i=1,\ldots,n$, are all distinct. 
 Writing $\diag(x_1,\dots,x_m)$ for the diagonal matrix with main diagonal $(x_1,\dots,x_m)$, we define 
\begin{align*}
\Lambda &:= \mathrm{diag}(\lambda_1,\ldots, \lambda_{k'}),\\
D_i &:= \mathrm{diag}(\mu_{i,1},\ldots, \mu_{i, {k_i}}),\;i=1,\dots,n,\\
D_a &:= \mathrm{diag}(a_1,\ldots, a_n),\\
D_\mu &:=D_1\oplus\dots\oplus D_n.
\end{align*}
By \cite[Theorem 4.2]{MS} 
and our strict eigenvalue interlacing requirement, for $i=1,\ldots,n$ there exist matrices: 
\[
M_i: = \begin{pmatrix}
 a_i & \mathbf{b}_i^T \\
 \mathbf{b}_i & D_i
\end{pmatrix}
\]
with eigenvalues $\lambda_1,\ldots , \lambda_{k_i+1}$, where $\mathbf{b}_i$ is a nowhere zero vector. 
Clearly, the distinct eigenvalues of $M:=M_1\oplus\dots\oplus M_n\oplus \Lambda$ are contained in $\{\lambda_1,\dots,\lambda_{k+1}\}$, and in particular, $q(M)\le k+1$. The same is true for the matrix:
\[
M': = \begin{pmatrix}
D_a & B^T &0 \\
B & D_\mu&0\\
0&0&\Lambda
\end{pmatrix},
\]
where $B=\bigoplus_{i=1}^n \mathbf{b}_i$, since $M'$ is permutationally similar to $M$. 

Observe that the $n\times n$ matrix $D_a$ and the  $m\times m$ matrix $D_\mu\oplus \Lambda$ are both diagonal matrices with distinct eigenvalues. By Theorem \ref{generic},  their spectra are generically realizable for $G$ and $H$, respectively. Since the $m\times n$ matrix $Y:=\left(\begin{smallmatrix}B \\ 0 \end{smallmatrix}\right)$ has no zero column, by generic realizability for $H$ there is an orthogonal matrix $V$ so that $V( D_\mu\oplus \Lambda) V^T \in S(H)$ and $VY$ is nowhere zero. Since $(VY)^T=Y^TV^T$ is nowhere zero and so has no zero column, by generic realizability for~$G$ there is an orthogonal matrix $U$ so that  $UD_aU^T \in S(G)$ and $UY^TV^T$ is nowhere zero. 
 Now
\[
(U\oplus V)M' (U^T \oplus V^T) = \begin{pmatrix}
UD_a U^T & UY^TV^T  \\
VYU^T & V(D_\mu\oplus \Lambda) V^T
\end{pmatrix}\in S(G\vee H),
\]
so $q(G\vee H)\le k+1$ as required.

To see that the second claim follows from the first, observe that if $\max\{|G|,|H|\}>1$, then $|G|+|H|>\min\{|G|,|H|\}+1$, so $k:=\left\lceil{\frac{|G|+|H|}{\min\{|G|,|H|\}+1}}\right\rceil-1$ is a positive integer, and if $|G|\le |H|$, then $|G|\le |H|\le k|G|+k+1$, so $q(G\vee H)\le k+1=\left\lceil{\frac{|G|+|H|}{\min\{|G|,|H|\}+1}}\right\rceil$. By symmetry, the same holds if $|H|\le |G|$. 
\end{proof}



We remark that the hypothesis $|G|\le |H|\le k|G|+k+1$ in Theorem~\ref{thm:general_upper_bound} cannot be relaxed in general, since if we take $k=1$ and $G$ and $H$ are trees with $|H| >k|G|+k+1=|G|+2$, then $q(G\vee H)>2=k+1$ by~\cite[Example~3.5]{levene2021paths}.

The upper bound of Theorem \ref{thm:general_upper_bound} is sharp when $H$ is a path, as shown below. 


\begin{corollary}\label{GjoinP}
If $m>1$ and $G$ is a connected graph with $|G|=n\leq m$, 
then
$$q(G\vee P_m)=\left\lceil \frac{n+m}{n+1}\right\rceil.$$ 
\end{corollary} 
\begin{proof}
Let $X$ be a matrix in $S(G\vee P_m)$. Since $X$ has an $m \times m$ principal submatrix corresponding to $P_{m}$, this submatrix must have distinct eigenvalues. By eigenvalue interlacing, the matrix $X$ can have maximum eigenvalue multiplicity at most $n+1$. Hence 
\begin{equation*}
q(X) \geq \left\lceil \frac{|G\vee P_m|}{M(X)}\right\rceil  \geq \left\lceil\frac{n+m}{n+1}\right\rceil.
\end{equation*}
The opposite inequality was established in Theorem~\ref{thm:general_upper_bound}.
\end{proof}


\begin{remark}\label{rk:paths}
In the case $G=P_n$ where $2\le n\le m$, the formula of Corollary~\ref{GjoinP} improves on the upper bound $q(P_n\vee P_m)\le \lceil \tfrac {n+m}2\rceil$ which follows from \cite[Corollary~49]{MR3665573}, since $P_n\vee P_m$ contains a Hamiltonian cycle. 
\end{remark}



\noindent
We conclude this section with 
a theorem which resolves a question from \cite[Remark 3.13]{levene2021paths}.

\begin{corollary}\label{cor:valuejoincompletepath}
 If $m,n \geq 2$, then
 \[q(K_n \vee P_m) = 
 \left\lceil \frac{n+m}{n+1} \right\rceil.
 \]
\end{corollary}

\begin{proof}
For $n \leq m$, this is a special case of Corollary~\ref{GjoinP}. For $n\geq m$, note that $\lceil \frac{m+n}{n+1}\rceil=2$. We know from  Theorem 5.2 in
\cite{MR3506498} that $q(K_n \vee P_n)=2$, and for $n> m$, it follows that $q(K_n \vee P_m)=2$ by applying the notion of join duplication (jdup) and the inequality presented in (\ref{eq:jdup}). \end{proof}


\section{Bordering}\label{sec:border}


Recall from the introduction that an \emph{$r$-bordering} of a symmetric $n \times n$ matrix $A$ is any symmetric $(n+r)\times (n+r)$ matrix $B$ which contains $A$ as its $n \times n$ trailing principal submatrix of $B$.
Building upon the classical results derived from Cauchy's interlacing inequalities that characterize all possible eigenvalues of a $1$-bordering of $A$, we aim to understand the fewest number of distinct eigenvalues possible for an $r$-bordering of $A$. 

 First we have a look at $1$-borderings, noting that any $r$-bordering of $A$ can be obtained by repeated $1$-bordering. 

\begin{theorem}\label{thm:1-bordering}
Let $A$ be an $n\times n$ symmetric matrix and $A'$ a $1$-bordering of $A$. The following statements are equivalent: 
\begin{enumerate}
\item\label{thm:1-bordering1} $\mathcal{N}$ is the set of distinct eigenvalues $\lambda$ of $A'$ that satisfy $m_{A'}(\lambda)=m_{A}(\lambda)+1$, and 
$\mathcal{R}_0$ is the set of distinct eigenvalues $\lambda$ of $A$ that satisfy $m_{A'}(\lambda)=m_{A}(\lambda)-1$.
\item\label{thm:1-bordering2}
$A'=\left(
\begin{array}{cc}
 \alpha & {\bf b}^TU_0^T \\
 U_0{\bf b} & A \\
\end{array}
\right)$
where $k:=|\mathcal{R}_0|$, $U_0$ is an $n\times k$ matrix with $U_0^T U_0=I_{k}$, and $U_0^TAU_0$ is a $k \times k$ diagonal matrix $D_0$ with distinct eigenvalues equal to $\mathcal{R}_0$. Further, ${\bf b} \in \mathbb{R}^k$ is a nowhere zero vector
so that the matrix $$B=\left(
\begin{array}{cc}
 \alpha & {\bf b}^T \\
 {\bf b} & D_0 \\
\end{array}
\right)$$
has eigenvalues $\mathcal{N}$. 
\end{enumerate}
If the above hold, then
$A'$ is similar to a matrix of the form $D_{\mathcal{N}} \oplus D_1$ for some diagonal matrix $D_1$ via an orthogonal similarity using
\begin{equation}\label{eq:Wsimilarity}
W=\left(
\begin{array}{cc}
 {\bf v}^T & 0 \\
 U_0V_0 & U_1 \\
\end{array}
\right),
\end{equation}
where 
$V=\left(
\begin{array}{c}
 {\bf v}^T \\
 V_0
\end{array}
\right)\in \bR^{|\mathcal{N}|\times |\mathcal{N}|}$ is an orthogonal matrix that
satisfies $V^T BV=D_{\mathcal{N}}$, and $U=\left(
\begin{array}{cc}U_0 & U_1\end{array}
\right)$ is an orthogonal matrix that satisfies $U^TAU=D_{0} \oplus D_1$. 
\end{theorem}

\begin{proof}
\begin{description}
    \item[$(1\Rightarrow 2)$]
Let $\lambda_1, \ldots, \lambda_q$ be the distinct eigenvalues of $A$ with multiplicities $m_i:=m_{A}(\lambda_i)$, $i=1,\ldots,q$, and let $U'$ be an orthogonal matrix that diagonalizes $A$, that is, $U'^TAU'=\oplus_{j=1}^q \lambda_j I_{m_j}$. Then for some $\alpha\in \bR$ and ${\mathbf a}\in \bR^n$, we have $$A'_1:=(1\oplus U'^T)A'(1\oplus U')=\left(\begin{array}{cc}
 \alpha & {\bf a}^T \\
 {\bf a} & \oplus_{j=1}^q \lambda_j I_{m_j} \\
\end{array}\right).$$  

Write ${\bf a}^T=\left(
\begin{array}{cccc}
 {\bf a}_1^T & {\bf a}_2^T & \cdots & {\bf a}_q^T \\
\end{array}
\right)$, where ${\bf a}_i\in \mathbb{R}^{m_i}$. 
Choose orthogonal matrices $Z_i\in \mathbb{R}^{m_i\times m_i}$ that satisfy $Z_i{\bf a}_i=b_i {\bf e}_1$, where $b_i \in \mathbb{R}$ and ${\bf e}_1$ denotes the basic unit vector in $\bR^{m_i}$ whose first element is equal to $1$. Note that $b_i\neq 0$ if and only if $\lambda_i \in \mathcal R_0$ (see for example \cite[Lemma 5.1]{MR3567513}  for the nontrivial implication).  
Applying the orthogonal similarity $1 \oplus (\oplus_{i=1}^q Z_i )$ to $A'_1$, followed by a permutation similarity $1 \oplus P$, we see that $A'$ is orthogonally similar to $B \oplus D_1$, where $D_1$ is a diagonal $(n-k)\times (n-k)$ matrix, \[B=\left(
\begin{array}{cc}
 \alpha & {\bf b}^T \\
 {\bf b} & D_0 \\
\end{array}
\right)\]
and ${\bf b}\in \bR^{k}$ is a nowhere zero vector. In summary, $U:=U'(\oplus_{i=1}^k Z_i^{T} )P$ satisfies $U^TAU=D_0 \oplus D_1$ and $(1 \oplus U)^TA'(1\oplus U)=B \oplus D_1$. 
Writing $U=\left(
\begin{array}{cc}
 U_0 & U_1 \\
\end{array}
\right)$ where $U_0\in \bR^{n\times k}$ and computing $A'=(1 \oplus U)(B \oplus D_1)(1\oplus U^T)$ gives the form for $A'$ as in item~\ref{thm:1-bordering2}.
\item[$(2\Rightarrow 1)$] Let $A'$ and $U_0$ be as in item \ref{thm:1-bordering2}, and $U_1 \in  \mathbb{R}^{n\times (n-k)}$ be such that $U:=\left(
\begin{array}{cc}
 U_0 & U_1 \\
\end{array}
\right)$ is orthogonal and
$ U^TAU$ is a diagonal matrix $D_0 \oplus D_1$. From
$$(1 \oplus U^T)A'(1\oplus U)=B \oplus D_1$$
we conclude that $A'$ has eigenvalues as stated in item \ref{thm:1-bordering1}.
\end{description}\medskip

To prove the final claim we note that:
$$W:=(1\oplus U)(V\oplus I_{n-k})=\left(
\begin{array}{cc}
 {\bf v}^T & 0 \\
 U_0V_0 & U_1 \\
\end{array}
\right)$$
and $W^TA'W=(V^T\oplus I)(B\oplus D_1)(V\oplus I)=D_{\mathcal{N}} \oplus D_1$, as claimed.
\end{proof}

Theorem \ref{thm:1-bordering} provides a construction of a $1$-bordering of a symmetric matrix, subject to quite general eigenvalue constraints. Our first application of this theorem produces a known result
\cite[Thm. 4.3.10]{HJ1}. We include it here mostly to establish notation that we will depend on in the rest of this section.


\begin{corollary}\label{lem:1-border}
Let $A$ be an $n\times n$ symmetric matrix, $\mathcal{R}$ the set of distinct eigenvalues of $A$, and $\mathcal{R}_0\subseteq \mathcal{R}$. If $\mathcal{N}$ is any set of $|\mathcal{R}_0|+1$ distinct real numbers which strictly interlaces $\mathcal{R}_0$, then there is a $1$-bordering $A'$ of $A$ so that for $\lambda \in \bR$,
\begin{align*}
m_{A'}(\lambda)=\begin{cases}m_{A}(\lambda)-1 &\text{if }\lambda \in \mathcal{R}_0, \\
m_{A}(\lambda)+1 &\text{if }\lambda \in \mathcal{N}, \\
m_{A}(\lambda) &\text{otherwise},\end{cases}
\end{align*}
where $m_{A'}(\lambda)=0$ means that $\lambda$ is not an eigenvalue of $A'$. 
\end{corollary}

\begin{proof}
Let $D_0$ be a diagonal matrix with distinct diagonal elements equal to elements in $\mathcal{R}_0$.
By \cite{boley-golub},
since $\mathcal{N}$ strictly interlaces $\mathcal{R}_0$,  there exist $a \in \mathbb{R}$ and a (nowhere zero) vector ${\bf b} \in \mathbb{R}^{|\mathcal{R}_0|}$  so that the matrix $$B=\left(
\begin{array}{cc}
 a & {\bf b}^T \\
 {\bf b} & D_0 \\
\end{array}
\right)$$
has the set of eigenvalues equal to $\mathcal{N}$. The result now follows from Theorem \ref{thm:1-bordering}.
%
\end{proof}

Starting with the eigenvalues of $A$, we will reduce the number of distinct eigenvalues of an $r$-bordering of $A$ by removing all eigenvalues from different intervals. Along these lines, we let $m_A(\alpha,\beta)$ denote the sum of multiplicities of all eigenvalues $\lambda$ of $A$ that are contained in the open interval $(\alpha,\beta)$, where $\alpha,
\beta\in\bR\cup \{-\infty,\infty\}$ with $\alpha<\beta$.

The following straightforward consequence of eigenvalue interlacing produces a lower bound on $r$ for an $r$-bordering to have no eigenvalues in a given interval. 
\begin{lemma}\label{lem:bordering-interlacing}
   If $M$ is an $r$-bordering of a symmetric matrix~$A$ and $\alpha,\beta\in \bR\cup \{-\infty,\infty\}$ with $\alpha<\beta$, then \[|m_A(\alpha,\beta)- m_M(\alpha,\beta)|\le r.\]
\end{lemma}
\begin{proof}
  The eigenvalues of $A$ and any $1$-bordering of $A$ must interlace by Cauchy's interlacing inequalities, which establishes the case $r=1$. In general, $M$ is obtained by $r$ successive $1$-borderings of $A$, and the statement follows immediately.
\end{proof}

Let ${\bf m}=(m_1,\ldots,m_k)\in \bN_0^k$ be an ordered multiplicity list of a symmetric matrix. 
For $2\leq t\leq k$, we define
\begin{equation}\label{eq:C(m,t)}
 C({\bf m},t)=\min_{1=p_1\le p_2\le \cdots \le p_t=k}\left( \max_{1\le i\le t-1}
g_{\bf m}
(p_i,p_{i+1})\right)\end{equation}
where \[
g_{\bf m}(p_i,p_{i+1}):=
\sum_{j=p_i+1}^{p_{i+1}-1}m_j .\]

In colloquial terms,
$C({\bf m},t)$ is the solution to the problem of minimizing the largest ``gap multiplicity'' $g_{\bf m}$ of ${\bf m}$, over the gaps given by the various choices of $t$ ``gap boundaries'' $1=p_1\le p_2\le \cdots\le p_t=k$. 

\begin{example}
To illustrate the preceding definition, we demonstrate how $C({\bf m},t)$ is computed for the case when $\mathbf{m}=(1,2,5,5,3,1)$ and $t=3$, by listing the possible values of the gap multiplicities for the various choices of gap boundary $p_2$ in Table~\ref{tab:gaps}. 
The minimum of the maximum gap multiplicities is 7, so $C(\mathbf{m},3)=7$.
One can also determine that $C(\mathbf{m},4)=3$, $C(\mathbf{m},5)=2$, and $C(\mathbf{m},2)=15$. 
\end{example}

\begin{table}[htb]
\begin{center}
\begin{tabular}{cccc}\hline
$p_2$ & $g_{\bf{m}}(p_1,p_2)$ & $g_{\bf{m}}(p_2,p_3)$ & \parbox{3cm}{\centering \strut maximum gap multiplicity\strut}\\ \hline
1 & $0$ & $\hspace{-5ex}m_2+m_3+m_4+m_5=2+5+5+3\hspace{-3ex}$ & 15\\
2 & $0$ & $m_3+m_4+m_5=5+5+3$ & 13\\
3 & $m_2=2$ & $m_4+m_5=5+3$ & 8 \\
4 & $m_2+m_3=2+5$ & $m_5=3$ & 7 \\
5 & $m_2+m_3+m_4=2+5+5$&  $0$ & 12 \\
6 & $m_2+m_3+m_4+m_5=2+5+5+3\hspace{-5ex}$&  $0$ & 15 \\ \hline
\end{tabular}
\caption{The list of possible values of $p_2$ and the corresponding parameters inside the formula for $C({\mathbf m},t)$ for ${\bf m}=(1,2,5,5,3,1)$ and $t=3$. Note that in each case, we have $p_1=1$ and $p_3=6$, whereas $p_2$ can vary. The list of multiplicities
in each of the two gaps derived from each value of $p_2$ and the corresponding maximum gap multiplicities are given.}\label{tab:gaps}
\end{center}
\end{table}

Note that $q({\bf m})\leq t$ if and only if $C({\bf m},t)=0$. Indeed, if $q({\bf m})\le t$, then we may assume that ${\bf m}\in \bN^k$ where $k=q({\bf m})\le t$, and then choosing $p_i=\min \{i,k\}$ in \eqref{eq:C(m,t)} shows that $C({\bf m},t)=0$; conversely, if $C({\bf m},t)=0$ is attained for some particular $1=p_1\le p_2\le\dots\le p_t=k$, then 
$m_i=0$ for all $i\in [k]\setminus \{p_1,\dots,p_t\}$, 
hence $q({\bf m})\le t$. Hence, we can view $C({\bf m},t)$ as a measure of how far the multiplicity list ${\bf m}$ is from having $q({\bf m})=t$. This will be made more precise in the next proposition. 

\begin{proposition}\label{prop:C-bordering}
Let $A$ be a symmetric matrix with $k\ge2$ distinct eigenvalues and ordered multiplicity list ${\bf m}=(m_1,\ldots,m_k)\in \bN^k$, and let $2\le t\le k$. For $r \in \bN_0$ the following statements are equivalent:
\begin{enumerate}
    \item there is an $r$-bordering $M$ of $A$ with $q(M)\le t$;
    \item $C({\bf m},t)\le r$.
\end{enumerate}
\end{proposition}
\begin{proof}Let $\lambda_1<\cdots<\lambda_k$ be the distinct eigenvalues of $A$, with $m_A(\lambda_i)=m_i$ for $i=1,\ldots,k$. 
\begin{description}
    \item[$(1\Rightarrow 2)$] 
    Suppose $\mu_1<\cdots<\mu_\tau$ are the distinct eigenvalues of some $r$-bordering $M$ of $A$, where $\tau\le t$. By eigenvalue interlacing, we have $\lambda_j\in [\mu_1,\mu_\tau]$ for every $j$. Hence, there is a unique $i_0$ with $\lambda_1\in [\mu_{i_0},\mu_{i_0+1})=[\nu_1,\nu_2)$, where $\nu_i:=\mu_{i_0-1+i}$. For $1\le i\le \tau-i_0$, define
  \[ p_i:=\min\{j\colon 1\leq j \leq k, \lambda_j\in [\nu_{i},\nu_{i+1})\} \]
  and let $p_{i}:=k$ for $i> \tau-i_0$. Then $1=p_1\le p_2\le\cdots\le p_t=k$. Moreover, if $p_i<j<p_{i+1}$, then $\lambda_j\in (\nu_{i},\nu_{i+1})$, so
  \[
  g_{\mathbf m}(p_i,p_{i+1})= \sum_{j:p_i<j<p_{i+1}} m_A(\lambda_j)\le m_A(\nu_{i},\nu_{i+1})\le r,\]
  where the final inequality follows from Lemma~\ref{lem:bordering-interlacing}, since $m_M(\nu_i,\nu_{i+1})=0$. Hence,
  \[ C({\bf m},t)\le \max_{1\le i\le t-1} 
  g_{\mathbf m}(p_i,p_{i+1})\le r,\]
  as required.
    \item[$(2\Rightarrow 1)$] If $C({\bf m},t)=0$, then $q(A) \leq t$ and we can take $r=0$. From now on we assume $C({\bf m},t)>0$.  Since $r\ge C({\bf m},t)$, there exist $p_1=1<p_2<\cdots<p_{t'}=k$ where ${t'}\le t$ so that 
  \[ m_A(\lambda_{p_i},\lambda_{p_{i+1}})=g_{\mathbf m}(p_i,p_{i+1})
  \le r,\quad 1\le i<{t'}.\]
  It suffices to find a $1$-bordering $M_1$ of $A$ so that $\sigma(M_1)\subseteq [\lambda_1,\lambda_k]$ and
  \[ m_{M_1}(\lambda_{p_i},\lambda_{p_{i+1}})\le \max\{r-1,0\},\quad 1\le i<{t'},\]
  since we can then continue inductively to find $A=M_0,M_1,\ldots,M_r=:M$, where $M_{\ell+1}$ is a $1$-bordering of $M_\ell$, so that $m_{M_r}(\lambda_{p_i},\lambda_{p_{i+1}})=0$ for $1\le i<{t'}$ and every eigenvalue of $M_r$ is in $[\lambda_1,\lambda_k]$, hence $M_r$ has only the ${t'}$ distinct eigenvalues $\{\lambda_{p_1},\ldots,\lambda_{p_{t'}}\}$.
  
  To show that such a matrix $M_1$ exists, first enumerate the open intervals $L_i:=(\lambda_{p_i},\lambda_{p_{i+1}})$ which contain at least one eigenvalue of $A$ as $L_{i_1},\ldots,L_{i_s}$, where $1\le i_1<\cdots<i_s<{t'}$, and choose $\mu_j\in \sigma(A)\cap L_{i_j}$ for $1\le j\le s$. (The assumption $C({\bf m},t)>0$ guarantees that at least one such interval exists.) Let $\mathcal{R}_0=\{\mu_1,\ldots,\mu_s\}$, and choose any set $\mathcal{N}\subseteq \{\lambda_{p_{1}},\ldots,\lambda_{p_{{t'}}}\}$ of size $s+1$ which strictly interlaces $\mathcal{R}_0$. The matrix constructed in Corollary~\ref{lem:1-border} then has the desired properties.
\qedhere\end{description}
\end{proof}

Given an $n \times n$ symmetric matrix $A$ with $\sigma(A)=\{\lambda_1^{(m_1)},\ldots,\lambda_k^{(m_k)}\}$, $\sum_{i=1}^{k} m_i=n$, the general procedure to find an $r$-bordering matrix $M$ of $A$ with $q(M) \leq t$ is shown in Algorithm~\ref{algo} below. Note that we may have some freedom in how we choose the sets ${\mathcal R}_0$ and ${\mathcal N}$ in each step. One possible choice is given in the proof of Proposition~\ref{prop:C-bordering}, and we show all possible choices for a particular case in Example~\ref{ex:options}.

\begin{algorithm}[ht!]
	\DontPrintSemicolon
	\SetKwFunction{Fstep}{Find an $r$-bordering matrix $M$ of $A$ with $q(M) \leq t$ }
	\SetAlgoLined	
\begin{enumerate}
    \item  Choose an integer $t$, $2\leq t\leq q(A)$. Define $M_{0}:=A$, $r:=C({\bf m},t)$. 
    \item For  $\ell=1,\ldots,r$, use  Corollary~\ref{lem:1-border} to construct an $(n+\ell)\times (n+\ell)$ matrix $M_{\ell}$  such that $$C({\bf m}(M_{\ell}),t)=C({\bf m}(M_{\ell-1}),t)-1.$$ Note that we may have some freedom in how we choose the sets ${\mathcal R}_0$ and ${\mathcal N}$ in each step.
    \item  The resulting $(n+r) \times (n+r)$ matrix $M:=M_{r}$ has $q(M) \leq t$.
\end{enumerate}	 
	\caption{\FuncSty{Find an $r$-bordering matrix $M$ of $A$ with $q(M) \leq t$ }}
	\label{algo}
\end{algorithm}

\section{Joins with complete graphs}\label{sec:joinswithcompletegraphs}

In this section we consider the join of two graphs and develop a technique for determining, under certain conditions, the minimum number of distinct eigenvalues for the join of a graph with a complete graph.

\subsection{Patterns and eigenvectors}
If we want a $1$-bordering of the matrix $A \in S(G)$ to produce a matrix $A' \in S(K_1 \vee G)$, then we need $U_0{\bf b}$ to have no zero entries in Theorem \ref{thm:1-bordering} above. This will happen for most choices of ${\bf b}$, unless $U_0$ contains a zero row, or equivalently, unless eigenvectors corresponding to the eigenvalues in $\mathcal R_0$ all have a zero entry in the same position. The next results consider the case $|\mathcal R_0|=1$. We call an eigenvalue of a symmetric matrix \emph{extreme} if it is the smallest or the largest eigenvalue of that matrix.

\begin{corollary}\label{cor:join-with-K1}
Suppose $G$ is a non-empty graph and there exists an $A\in S(G)$ with a nowhere zero eigenvector associated with some eigenvalue $\lambda$ of $A$. Then there exists a $1$-bordering $A'$ of $A$ in $S(K_1 \vee G)$ so that:
\begin{itemize}
\item $q(A')=q(A)+1$ if $\lambda$ is an extreme eigenvalue, 
\item $q(A')=q(A)$ if $\lambda$ is not an extreme eigenvalue,  
\item $q(A')=q(A)-1$ if $\lambda$ is simple and not an extreme eigenvalue. 
\end{itemize}
\end{corollary}

\begin{proof}
In Theorem \ref{thm:1-bordering} we choose $\mathcal R_0=\{\lambda\}$, $U_0\in\bR^{n\times 1}=\bR^n$ a nowhere zero eigenvector of $A$ with eigenvalue $\lambda$, and $B$ with eigenvalues $\mu_1$, $\mu_2$, satisfying $\mu_1< \lambda<\mu_2$, so that either $\mu_1$ or $\mu_2$ agrees with an eigenvalue of $A$, if $\lambda$ is an extreme eigenvalue, and so that both $\mu_1$ and $\mu_2$ are eigenvalues of $A$, if $\lambda$ is not an extreme eigenvalue of $A$. Since $U_0$ is a single column with no zero entries we get $A' \in S(K_1 \vee G)$, and since the spectrum of $A'$ can be obtained for the spectrum of $A$ by removing one multiple of $\lambda$ and increasing the multiplicity of $\mu_1$ and $\mu_2$ by~$1$, the result follows. 
\end{proof}

In Theorem \ref{thm:1-bordering} we have seen that after $1$-bordering, some eigenvectors will necessarily have a zero entry, and this has an interesting consequence for the patterns of $2$-borderings.

\begin{corollary}\label{cor:up and down}
Let $A$ be a symmetric matrix, $A'$ a $1$-bordering of $A$, and $A''$ a $1$-bordering of~$A'$.
%
If $(A'')_{1,2}\ne 0$, 
then there is an eigenvalue $\lambda$ of $A'$ so that 
$m_{A''}(\lambda)=m_{A'}(\lambda)-1=m_A(\lambda)$.
\end{corollary}

\begin{proof}
Adopting the notation and definitions from Theorem \ref{thm:1-bordering}, recall that $W=(W_{\mathcal N}\; W_1)$ where
$$
W_{\mathcal N}=\left(
\begin{array}{c}
 {\bf v}^T \\
 U_0V_0 \\
\end{array}
\right)\quad\text{and}\quad  W_1=\left(
\begin{array}{c}
 0 \\
 U_1 \\
\end{array}
\right),$$
and $W^TA'W=D_{\mathcal N}\oplus D_1$. Hence, $A'W=W(D_{\mathcal N}\oplus D_1)$, i.e., $(A'W_{\mathcal N}\;A'W_1)=(W_{\mathcal N}D_{\mathcal N}\;W_1D_1)$, so the columns of the matrices $W_{\mathcal{N}}$ and $W_1$
are eigenvectors of $A'$ corresponding to the eigenvalues of $D_\mathcal{N}$ and $D_1$, respectively. If $\lambda$ is an eigenvalue of $A'$ which is not in $\mathcal{N}$, then the $\lambda$-eigenspace of $A'$ is contained in the column space of $W_1$, so every vector in this eigenspace has first entry equal to zero.
It follows that any
eigenvector of $A'$ with nonzero first entry must have its corresponding eigenvalue $\lambda$ in $\mathcal{N}$.

Consider now the $1$-bordering $A''$ of $A'$. Let us define ${\mathcal R}'_0$, $D_0'$, $U'_0$ and ${\bf b}'$ for this $1$-bordering, analogously as was done above for the $1$-bordering $A'$ of $A$. If $(A'')_{1,2}\ne 0$, then $(U'_0 {\bf b}')_1 \ne 0$ by the above, so the first row of $U'_0$ cannot be a zero row. Since ${U'_0}^T A' U'_0=D_0'$, this implies that there is some eigenvector of $A'$, with eigenvalue $\lambda \in {\mathcal R}'_0$, which has a nonzero first entry. Hence, by the previous paragraph, $\lambda \in \mathcal N \cap \mathcal R_0'$, and thus  
$m_{A''}(\lambda)=m_{A'}(\lambda)-1=m_A(\lambda)$. 
\end{proof}

\begin{remark}\label{rk:up-and-down}
Suppose $r\ge2$ and $A_0,A_1,\dots,A_r$ are successive $1$-borderings of a matrix $A_0\in S(G)$. If $A_r\in S(K_r\vee G)$, then by Corollary~\ref{cor:up and down}, it is necessarily the case that for $0\le s\le r-2$, there is a real number $\lambda_s$ so that $m_{A_{s+2}}(\lambda_s)=m_{A_{s+1}}(\lambda_s)-1=m_{A_s}(\lambda_s)$.
\end{remark}

In the following example we illustrate how Algorithm~\ref{algo} may be used to border a matrix achieving a small $q$ value in $3$-bordering in different ways. We also identify cases when Remark~\ref{rk:up-and-down} implies that the resulting $3$-bordering cannot be in $S(K_3\vee G)$. 

\begin{example}\label{ex:options}
 Let $A$ be a $9 \times 9$ symmetric matrix with ordered multiplicity list ${\bf m}=(1,3,3,1,1)$ and spectrum $\{1,2^{(3)},3^{(3)},4,5\}$. The goal is to find the spectra of all $3$-borderings of $A$ that have three distinct eigenvalues. Observe that $C({\bf m},3)=3$, and to achieve this goal the value of $C$ must decrease by one every time we border. Table~\ref{tab:13311} shows all possible eigenvalues we can obtain in 
 this way. We produced this table by exhaustive search.  
 
 \begin{table}[htb]
     \centering
     \begin{tabular}{|c|c|c|c|}
          \hline
          $A$& \multicolumn{3}{c|}{$\{1,{\color{red} 2^{(3)}},3^{(3)},4,5\}$} \\        
          \hline
          $1$-bordering& $\{1^{(2)},{\color{red} 2^{(2)}},3^{(4)}, 5,{\color{blue} \mu}\}$&$\{1^{(2)},{\color{red} 2^{(2)}},3^{(3)},{\color{blue}\lambda},\mu,\nu\}$&$\{1^{(2)},{\color{red} 2^{(2)}},3^{(3)},{\color{blue}\lambda},5^{(2)}\}$\\
          \hline
          $2$-bordering& $\{1^{(3)},{\color{red} 2},3^{(5)},{\color{SpringGreen4}\mu'},\mu''\}$&$\{1^{(3)},{\color{red} 2},3^{(4)},{\color{SpringGreen4} \rho},\mu^{(2)}\}$&$\{1^{(3)},{\color{red} 2},3^{(3)},{\color{SpringGreen4}\lambda'},5^{(3)}\}$\\
          \hline
          $3$-bordering& $\{1^{(4)},3^{(6)},\mu''^{(2)}\}$&$\{1^{(4)},3^{(5)},\mu^{(3)}\}$&$\{1^{(4)},3^{(4)},5^{(4)}\}$\\
          \hline
     \end{tabular}
     \caption{Red eigenvalues are the ones that are forced to have reduced  multiplicity in the next bordering, the blue ones satisfy the conclusion of Corollary~\ref{cor:up and down} for the $2$-bordering of $A$, and the green ones satisfy the same condition when we consider instead the $3$-bordering of $A$. Moreover, $\lambda,\lambda'\in [3,4]$, $\nu\in[4,5]$, $\mu, \mu',\mu''\geq 5$ and $\rho\in (3,\mu)$, are arbitrary.}
     \label{tab:13311}
 \end{table}
 We note that the construction in the proof of Proposition~\ref{prop:C-bordering} produces only the spectrum  $\{1^{(4)},3^{(6)},5^{(2)}\}$, which we obtain after $1$-bordering with spectrum $\{1^{(2)},2^{(2)},3^{(4)},4,5\}$ and  $2$-bordering with spectrum $\{1^{(3)},2,3^{(5)},4,5\}$. This example shows that there may be several options of choosing appropriate ${\mathcal N}$ and ${\mathcal R}_0$ sets in each step as we develop an $r$-bordering with the desired number of distinct eigenvalues.  

 In all three situations (corresponding to three columns of Table \ref{tab:13311}), if $A\in S(G)$, by appropriately choosing the free parameters, it is possible to satisfy the necessary conditions of Remark~\ref{rk:up-and-down} for the $3$-bordering of $A$ to be in $S(K_3 \vee G)$. However, if, for example, we choose $\lambda=4$ or $\lambda'=\lambda$ in the last column, then the conditions of the remark do not hold. 
\end{example}

\subsection{Hypercubes}\label{hyper}

In this section we explore the minimum number
of distinct eigenvalues for joins of complete graphs with a hypercube graph. 
Recall that for $t\ge 1$, the vertices of the hypercube graph $Q_t$ are the $2^t$ binary strings of length $t$, and its edges are the pairs of vertices with Hamming distance one. 
It was shown in 
\cite[Corollary~6.9]{MR3118943}
that if $Q_t$ is the hypercube graph with $t\geq 2$, then $q(Q_t)=2$. In the following, we use the matrix construction from  
\cite{MR3118943}
to demonstrate that $Q_t$ has a realization $A$ having $q(A)=2$ 
and a nowhere zero eigenvector.

\begin{theorem}\label{lemma:hypercube}
  For any two positive integers $s$ and $t$,
  $$q(K_s\vee Q_{t}) \leq 3.$$
  Moreover, if $s\le t$, then
  \[ q(K_s\vee Q_{2t+2})=3.\]
 \end{theorem}
 \begin{proof}
 We will demonstrate that $Q_t$ has
 a realization $A$ having $q(A)=2$ and
 a nowhere zero eigenvector.
 Corollary~\ref{cor:join-with-K1} will then imply that $q(K_1 \vee Q_t)\leq 3$ and so the result follows from the inequality~(\ref{eq:jdup}).
 As observed in \cite{MR3118943}, for any nonzero $\alpha$ and $\beta$ with
 $\alpha^2 +\beta^2=1$, $Q_t$ has a realization 
 \[ B=\left( \begin{array}{cc}
\alpha A & \beta I\\ 
 \beta I & -\alpha A
 \end{array} \right) \]
 such that $A^2=I$ and $q(B)=2$. The vector 
\[\begin{pmatrix} 
(I+\alpha A) \mathbf{1} \\ \beta \mathbf{1}
\end{pmatrix}\] 
with $\mathbf{1}$ representing the
 all ones vector, will be a nowhere zero eigenvector of $B$ with eigenvalue 1 for any $\alpha$ sufficiently small.

 

   The second part of the statement is a generalization of~\cite[Proposition~5.1]{MR3904092}. It uses~\cite[Theorem~1.9]{MR3904092}, which is a small correction of~\cite[Theorem~4.4]{MR3118943}.
      For $i=1,2,\ldots,t+1$, consider the vertices of the hypercube $Q_t$ given by the binary strings $v_i=00\cdots01100\cdots0$, with the two ones in  positions $2i-1$ and $2i$. Then $\{v_1,\ldots,v_{t+1}\}$ is a set of $t+1$ independent vertices in $K_s\vee  Q_{2t+2}$, and $N(v_i)\cap N(v_j)=V(K_s)$ for $i\ne j$. Therefore
   $$\left\vert \bigcup_{i \ne j} N(v_i)\cap N(v_j)\right\vert =s < t+1,$$
   hence $q(K_s\vee Q_{2t+2}) \geq 3$ by~\cite[Theorem~4.4]{MR3118943}.
 \end{proof}
 

 By Theorem \ref{thm:general_upper_bound},  if $s$ is chosen sufficiently large, then $q(K_s\vee Q_t)=2$.  Thus, in light of Theorem~\ref{lemma:hypercube}, 
 and the fact that $q(K_s\vee Q_t)$ is a non-increasing function of $s$ as
 per Equation~(\ref{eq:jdup}), it is natural to ask the following question: What is the minimum $s$ for which $q(K_s\vee Q_t)=2$?


 


\subsection{Cycles}\label{sec:applicationjoinscycles}

Given $A\in S(H)$ and a graph $G$, let $S(G\vee A)$ be the set of all matrices $X\in S(G\vee H)$ so that $X[H]=A$, and let $q(G\vee A)$ be the minimum $q(X)$ over all such matrices~$X$. Note that $q(G\vee A)\geq q(G\vee H)$. Suppose $A$ has ordered multiplicity list ${\bf m}={\bf m}(A)$. Given a number $t\ge2$ and a graph $G$, we want to determine whether or not $q(G\vee A)\le t$. By Proposition~\ref{prop:C-bordering}, a necessary condition is that \[C({\bf m},t)\le |G|.\]
In Section \ref{sec:limitations} we will show that this condition is not sufficient in general, since it may happen that none of the $|G|$-borderings guaranteed by Proposition~\ref{prop:C-bordering} has the correct graph, $G\vee H$, where for any $n\times n$ symmetric matrix $A=(a_{ij})$, $H=G(A)$ is defined as the graph on $n$ vertices with edges $\{i,j\}$ whenever $i \neq j$ and $a_{ij}\neq 0$.
In fact, it is not generally sufficient even in the case that $G$ is a complete graph. Despite this, we provide examples when the procedure from Section \ref{sec:border} is applied successfully.

 

Note that the necessary condition above may be written as 
\begin{equation}\label{eq:q(HvA)>=C} q(G\vee A)\ge \min\{t\ge2 : C({\bf m}(A),t)\le |G|\}.\end{equation}



Turning to cycles, it is known by~\cite[Theorem~3.4]{levene2021paths} that $q(K_{2k-2}\vee C_{2k})=2$. Next, we use the following result on the inverse eigenvalue problem for cycles to determine the minimum number of eigenvalues allowed for joins of complete graphs with even cycles.  

 
\begin{proposition}\label{IEPG cycles}(IEPG for cycles~\cite{MR2549049}).
Nonincreasing real numbers $\lambda_1\geq \cdots \geq \lambda_n$ are the eigenvalues of $A\in S(C_n)$ if and only if either
$$\lambda_1\geq \lambda_2>\lambda_3\geq \lambda_4 > \lambda_5\geq\cdots$$
or
$$\lambda_1> \lambda_2\geq \lambda_3> \lambda_4 \geq \lambda_5>\cdots.$$
Hence, if $k\geq 2$, $q(C_{2k})=k$ and $M(C_{2k})=2$.
\end{proposition}

Observe that if $\lambda$ is a multiple eigenvalue of $A \in S(C_n)$, then the multiplicity of $\lambda$ is two and there exists a nowhere zero eigenvector for $\lambda$ associated with $A$. If the latter did not hold then every eigenvector ${\bf x}$ for $\lambda$ would satisfy ${\bf x}_i=0$ for some $i=1,2,\ldots, n$. In this case $\lambda$ is a multiple eigenvalue for the principal submatrix of $A$ obtained by deleting row and column $i$. However, this submatrix lies in $S(P_{n-1})$, and can only possess simple eigenvalues.

\begin{theorem} If $k\geq 2$ then 
$q(K_1\vee C_{2k})=k.$
\end{theorem}
\begin{proof}
To obtain the upper bound $q(K_1\vee C_{2k}) \leq k$, use Proposition~\ref{IEPG cycles} to choose a matrix $A\in S(C_{2k})$ with multiplicity list $(2,2,\dots,2)$, choose a non-extreme eigenvalue of $A$ and a nowhere zero eigenvector and apply Corollary~\ref{cor:join-with-K1}.

To show the lower bound, assume that $M \in S(K_1\vee C_{2k})$ has eigenvalues $\mu_1 \leq \mu_2 \leq \cdots \leq \mu_{2k+1}$ and that $A$ is the submatrix corresponding to $C_{2k}$ and has eigenvalues $\lambda_1 \leq \cdots \leq \lambda_{2k}$. By Proposition \ref{IEPG cycles}, we have that the maximum multiplicity of an eigenvalue $\lambda_i$ is $2$ and furthermore, if there are eigenvalues $\lambda_i$ and $\lambda_j$ with multiplicity $2$ then $m_A(\lambda_i, \lambda_j)$ must be even. By eigenvalue interlacing we have that the maximum multiplicity of any eigenvalue of $M$ is $3$. We claim that if $\mu_i$ and $\mu_j$ each have multiplicity $3$, then there must be an eigenvalue of multiplicity $1$ between them, and the lower bound follows once we show this.

By way of contradiction, assume that there is some pair of eigenvalues with multiplicity $3$ and $j$ distinct eigenvalues between them, each with multiplicity $2$ (with the possibility that $j$ is $0$). That is, we have 
$$\mu_i = \mu_{i+1} = \mu_{i+2} < \cdots < \mu_{i+2+2j+1} = \mu_{i+2+2j+2} = \mu_{i+2+2j+3}.$$
From eigenvalue interlacing we must have $\lambda_i = \lambda_{i+1}$ and $\lambda_{i+2j+3} = \lambda_{i+2j+4}$. Hence it follows that both $\lambda_{i+1}$ and $\lambda_{i+2j+3}$ have multiplicity $2$ and $m_A(\lambda_{i+1}, \lambda_{i+2j+3}) = 2j+1$ is odd, a contradiction.
\end{proof}

 As we saw in Example~\ref{ex:options}, we have to be careful about the choice of $1$-bordering of $A$ in order to ensure that a subsequent $2$-bordering of $A$ has the desired pattern. As another illustration of this issue, observe that if an eigenvalue $\lambda$ of $A \in S(C_6)$ has multiplicity $2$ and multiplicity $1$ for a $1$-bordering $A'$ of $A$, then eigenvectors of $\lambda$ for $A'$ will not be nowhere zero---since the interlacing is not strict, an entry of the eigenvector for $A'$ is 0 (see \cite[Theorem 4.3.17]{HJ1}). 
 This shows that if we apply Algorithm \ref{algo} starting with $A \in S(C_6)$ with multiplicity list $(2,2,2)$ to produce a $2$-bordering $A''$ with $q(A'')=2$, then $A''\not\in S(K_2 \vee C_6)$. In the next example we show that starting with a matrix $A \in S(C_6)$ with a different multiplicity list, $q(A'')=2$ can still be reached for some $A''\in S(K_2 \vee C_6)$.

 \begin{example} Let
   \[A=\left(
\begin{array}{rrrrrr}
 1 & 1 & 0 & 0 & 0 & -1 \\
 1 & -1 & 1 & 0 & 0 & 0 \\
 0 & 1 & 1 & 1 & 0 & 0 \\
 0 & 0 & 1 & -1 & 1 & 0 \\
 0 & 0 & 0 & 1 & 1 & 1 \\
 -1 & 0 & 0 & 0 & 1 & -1 \\
\end{array}
\right)\in S(C_6),\quad U_0=\frac1{\sqrt 3}\left(
\begin{array}{rr}
 1 & 0 \\
 0 & 1 \\
 -1 & 0 \\
 0 & -1 \\
 1 & 0 \\
 0 & 1 \\
\end{array}
\right).\]
Then $\sigma(A)=\{(-2)^{(2)},-1,1,2^{(2)}\}$. Let $\mathcal R_0=\{-1,1\}$, and observe that $U_0^TAU_0=\left(
\begin{smallmatrix}
  1&0\\0&-1
\end{smallmatrix}\right)$, corresponding to the setup of Theorem~\ref{thm:1-bordering}.
Choose any $t\in (-1,1)$ and let
$\mathcal N=\{-2,t,2\}$. Following Corollary~\ref{lem:1-border} and~\cite[equation (2.4)]{boley-golub}, we find
\[
B= \left(
\begin{array}{rrr}
 t & \sqrt{3}u & \sqrt{3}v \\
 \sqrt{3}u &1 & 0 \\
 \sqrt{3}v & 0 & -1 \\
\end{array}
\right)\]where
\[ u=\sqrt{(1-t)/2}\quad\text{and}\quad v=\sqrt{(1+t)/2},\]
and
\[ A'=\left(
\begin{array}{rrrrrrr}
 t & u & v & -u & -v & u & v \\
 u & 1 & 1 & 0 & 0 & 0 & -1 \\
 v & 1 & -1 & 1 & 0 & 0 & 0 \\
 -u & 0 & 1 & 1 & 1 & 0 & 0 \\
 -v & 0 & 0 & 1 & -1 & 1 & 0 \\
 u & 0 & 0 & 0 & 1 & 1 & 1 \\
 v & -1 & 0 & 0 & 0 & 1 & -1 \\
\end{array}
\right)\in S(K_1\vee C_6)\]
with $\sigma(A')=\{(-2)^{(3)},t,2^{(2)}\}$.

Repeating the construction, choosing $\mathcal R_0'=\{t\}$ and $\mathcal N'=\{-2,2\}$, we obtain
\[ A''=\left(
\begin{array}{rrrrrrrr}
  -t & \sqrt{1-t^2} & -v & u & v & -u & -v & u \\[3pt]
 \sqrt{1-t^2} & t & u & v & -u & -v & u & v \\
 -v & u & 1 & 1 & 0 & 0 & 0 & -1 \\
 u & v & 1 & -1 & 1 & 0 & 0 & 0 \\
 v & -u & 0 & 1 & 1 & 1 & 0 & 0 \\
 -u & -v & 0 & 0 & 1 & -1 & 1 & 0 \\
 -v & u & 0 & 0 & 0 & 1 & 1 & 1 \\
 u & v & -1 & 0 & 0 & 0 & 1 & -1 \\
\end{array}
\right)\in S(K_2\vee C_6)\]
with $\sigma(A'')=\{(-2)^{(4)},2^{(4)}\}$ and thus $q(K_2\vee C_6)=2$.
 \end{example}
 
 \begin{example}
 Using the Jacobi-Ferguson algorithm \cite{fixme}, we can construct numerical matrices $A\in S(C_{10})$ with spectrum $\{(-6)^{(2)},-4,-2,0^{(2)},4,2,6^{(2)}\}$, and hence find numerical matrices $A''\in S(K_2\vee C_{10})$ with spectrum $\{(-6)^{(4)},0^{(4)},6^{(4)}\}$ and thus $q(K_2\vee C_{10})=3$. One such numerical matrix is:

{\tiny{
 \[A''=\left(
\begin{smallmatrix}
 0 & -1.9720 & -0.11321 & -0.40399 & -2.4521 & -1.3819 & 0.0061884 & 0.00028437 & -0.0036489 & 2.1264 & -2.9646 & 1.3043 \\
 -1.9720 & 0 & -2.2195 & -2.0495 & 2.2752 & -0.83772 & 0.0080390 & 0.005574 & -0.018511 & -1.9731 & -1.7971 & 1.6944 \\
 -0.11321 & -2.2195 & 0 & 3.2468 & 0 & 0 & 0 & 0 & 0 & 0 & 0 & 3.6901 \\
 -0.40399 & -2.0495 & 3.2468 & 0 & 3.6175 & 0 & 0 & 0 & 0 & 0 & 0 & 0 \\
 -2.4521 & 2.2752 & 0 & 3.6175 & 0 & 1.5399 & 0 & 0 & 0 & 0 & 0 & 0 \\
 -1.3819 & -0.83772 & 0 & 0 & 1.5399 & 0 & 0.010306 & 0 & 0 & 0 & 0 & 0 \\
 0.0061884 & 0.0080390 & 0 & 0 & 0 & 0.010306 & 0 & 5.4891 & 0 & 0 & 0 & 0 \\
 0.00028437& 0.005574 & 0 & 0 & 0 & 0 & 5.4891 & 0 & 2.4227 & 0 & 0 & 0 \\
 -0.0036489 & -0.018511 & 0 & 0 & 0 & 0 & 0 & 2.4227 & 0 & 0.013409 & 0 & 0 \\
 2.1264 & -1.9731 & 0 & 0 & 0 & 0 & 0 & 0 & 0.013409 & 0 & 2.9999 & 0 \\
 -2.9646 & -1.7971 & 0 & 0 & 0 & 0 & 0 & 0 & 0 & 2.9999 & 0 & -2.7171 \\
 1.3043 & 1.6944 & 3.6901 & 0 & 0 & 0 & 0 & 0 & 0 & 0 & -2.7171 & 0 \\
\end{smallmatrix}
\right).\]
}}
  \end{example}
 

\section{Limitations of Algorithm~\ref{algo} for graph joins}\label{sec:limitations}

In this section we show that the condition $C({\bf m}(A),t)\le r$ from Proposition \ref{prop:C-bordering} is not generally sufficient in the case $G=K_r$ for the existence of a matrix with at most $t$ eigenvalues in $S(G\vee A)$.

\begin{proposition}\label{prop:monotone}
 Suppose $t\ge2$ and $A_1,\ldots,A_k$ are successive $1$-borderings of a symmetric matrix $A=A_0$, and ${\bf m}(A)=(m_1,k,m_2,k,\ldots,k,m_t)$ where $m_j\ge k\ge t$ for each $j$, and $q(A_k)=t$. Then \[{\bf m}(A_j)=(m_1+j,k-j,m_2+j,k-j,\ldots,k-j,m_t+j),\quad j=0,1,\ldots,k.\]
\end{proposition}
\begin{proof}
  Let $\mu_1<\cdots<\mu_{2t-1}$ be the distinct eigenvalues of $A_0$ and $\lambda_1<\cdots<\lambda_t$ the distinct eigenvalues of $A_k$. By eigenvalue interlacing, every eigenvalue of $A_j$ is in the closed interval $[\lambda_1,\lambda_t]$. Moreover, by Lemma~\ref{lem:bordering-interlacing}, for $i=1,\ldots,t-1$ and $j=0,\ldots,k$, we have \[m_{A_j}(\lambda_i,\lambda_{i+1})\ge m_{A_0}(\lambda_i,\lambda_{i+1})-j.\] In particular, $0=m_{A_k}(\lambda_i,\lambda_{i+1})\ge m_{A_0}(\lambda_i,\lambda_{i+1})-k$, so \[m_{A_0}(\lambda_i,\lambda_{i+1})\le k.\] Let $S=\{\mu_1,\ldots,\mu_{2t-1}\}\setminus\{\lambda_1,\ldots,\lambda_t\}$. Then $|S|\ge 2t-1-t=t-1$, and each eigenvalue in $S$ has multiplicity at least $k$ in $A_0$ by hypothesis, so
  \[
  k(t-1)\le k|S| \le \sum_{i=1}^{t-1} m_{A_0}(\lambda_i,\lambda_{i+1})\le k(t-1).
  \]
  Hence, $|S|=t-1$, so $\{\lambda_1,\ldots,\lambda_t\}\subseteq\{\mu_1,\ldots,\mu_{2t-1}\}$. Since $\mu_1,\mu_{2t-1}\in [\lambda_1,\lambda_t]$, this forces $\mu_1=\lambda_1$ and $\mu_{2t-1}=\lambda_t$. If $\lambda_i=\mu_j$ and $\lambda_{i+1}=\mu_l$ where $l>j+2$, then $m_{A_0}(\lambda_i,\lambda_{i+1})\ge 2k$, a contradiction. It follows that $\lambda_i=\mu_{2i-1}$ for $1\le i\le t$.
  
  Hence, $m_{A_0}(\lambda_i,\lambda_{i+1})=k$ for each $i$, and the bound we observed above becomes 
  \[ k-j\le m_{A_j}(\lambda_i,\lambda_{i+1}).\]
  Since $A_k$ is a $(k-j)$-bordering of $A_j$, by Lemma~\ref{lem:bordering-interlacing} we also have
  \[ m_{A_j}(\lambda_i,\lambda_{i+1})\le m_{A_k}(\lambda_i,\lambda_{i+1})+k-j=k-j,\]
  so $m_{A_j}(\lambda_i,\lambda_{i+1})=k-j$. Moreover, by eigenvalue interlacing, $k-j\le m_{A_j}(\mu_{2i})\le m_{A_j}(\lambda_i,\lambda_{i+1})=k-j$, so we have equality. Hence, the multiplicity of $\mu_{2i}$ as an eigenvalue of $A_j$ is $k-j$, and no other real number in $(\lambda_i,\lambda_{i+1})$ is an eigenvalue of $A_j$. It follows that every eigenvalue of $A_j$ other than $\mu_2,\ldots,\mu_{2(t-1)}$ is in the set $\{\lambda_1,\ldots,\lambda_t\}$. Observe that $A_j$ is an $(j+N)\times (j+N)$ matrix, where $N=(t-1)k+\sum_{i=1}^t m_i$ is the number of rows and columns of $A$. Hence,
  \begin{align*}
     \sum_{i=1}^tm_{A_j}(\lambda_i) = j+N-\sum_{i=1}^{t-1}m_{A_j}(\mu_{2i})=j-(t-1)(k-j)+(t-1)k+\sum_{i=1}^tm_i=\sum_{i=1}^t(m_i+j).
  \end{align*}
  Since the total multiplicity of the eigenvalues $\lambda_1,\dots,\lambda_t$ in $A_j$ is $\sum_{i=1}^t(m_i+j)$, and by eigenvalue interlacing, the multiplicity in $A_j$ of $\lambda_i=\mu_{2i-1}$ is bounded above by $m_i+j$, this must be precisely its multiplicity.
 %
\end{proof}

\begin{corollary}\label{coro:limitationsnotsufficientcondition}
If $A$ is a symmetric matrix with 
${\bf m}(A)=(m_1,k,m_2,k,\ldots,k,m_t)$ where $m_i\ge k\ge t\ge 2$ for each $i$, then $C({\bf m}(A),t)=k$ yet $q(G\vee A)>t$ for all non-empty graphs $G$ with $|G|=k$. Hence, the inequality~\eqref{eq:q(HvA)>=C} is strict in this case.
\end{corollary}
\begin{proof}
    We have $C({\bf m}(A),t)=k$, so $q(B)\ge t$ for all $k$-borderings $B$ of $A$ by Proposition~\ref{prop:C-bordering}.
    Consider a sequence of successive $1$-borderings taking us from $A$ to some $k$-bordering $B$ with $q(B)=t$. By Proposition~\ref{prop:monotone}, the successive eigenvalue multiplicities of any given $\lambda\in \bR$ in this sequence of matrices is monotone. Hence, by Corollary~\ref{cor:up and down}, the superdiagonal of the leading principal $k\times k$ submatrix of $B$ is zero.
    
    Now let $P$ be a $k\times k$ permutation matrix, and consider $B_P=(P\oplus I_r)B(P^T\oplus I_r)$, where $k+r=|G|$. By the previous paragraph, the superdiagonal of the leading principal $k\times k$ submatrix of $B_P$ is zero, for every such permutation matrix~$P$. 
    Hence, 
    every off-diagonal entry of $B$ is zero, so $B$ has an empty graph. 
\end{proof}

This shows a limitation of Algorithm~\ref{algo}.
However, we show in the following proposition that this limitation is very specific, and that if the multiplicity list is perturbed only slightly we may have success using this procedure.

\begin{proposition}
Suppose $t\ge2$ and $A$ is a symmetric matrix with eigenvalues 
\[\lambda_1^{(m_1)} < \beta < \gamma < \lambda_2^{(m_2)} < \mu_2^{(2)} < \lambda_3^{(m_3)}< \mu_3^{(2)} < \cdots < \mu_{t-1}^{(2)} < \lambda_t^{(m_t)}.
\]
If $A$ has an eigenbasis such that for each vertex $u$ there is at least one eigenvector corresponding to an eigenvalue in $\{\mu_i\}$ which is nonzero in the entry corresponding to $u$, then there exists a matrix $B\in S(K_2\vee A)$ 
such that $B$ has eigenvalues $\lambda_1^{(m_1+2)},\ldots , \lambda_t^{(m_t+2)}$. In particular, $q(K_2 \vee G(A)) \leq t$.
\end{proposition}


\begin{proof}
By~\cite{boley-golub} we know that there are $1$-borderings $C_\beta$ and $C_\gamma$ of the matrices $\mathrm{diag}(\beta, \mu_2,\ldots, \mu_t)$ and $\mathrm{diag}(\gamma, \mu_2,\ldots, \mu_t)$ respectively which each have eigenvalues $\{\lambda_1,\ldots, \lambda_t\}$. Furthermore, we know that these borderings can have no zeros in the first row or column, and by computing traces we see that the $(1,1)$ entries are $k-\beta$ and $k-\gamma$ respectively, where $k=\lambda_1+\dots+\lambda_t-(\mu_2+\dots+\mu_t)$. Let the first row of $C_\beta$ have entries $k-\beta, b_1, \ldots, b_{t-1}$ and the first row of $C_\gamma$ have entries $k-\gamma, c_1,\ldots, c_{t-1}$. Define $B_0=[v_\beta,v_\gamma]$ where $v_\beta=(k-\beta,0,b_1,0,b_2,0,\dots)^T$, and $v_\gamma=(0,k-\gamma,0,c_1,0, c_2,\dots)^T$. That is, we are making vectors with the first rows of the borderings in the even or odd positions. Now define matrices $D_1=\diag(k-\beta, k-\gamma)$, $D_2=\diag(\beta, \gamma, \mu_2^{(2)},\ldots, \mu_{t-1}^{(2)})$ and $D_0=\diag(\lambda_1^{(m_1)},\ldots, \lambda_t^{(m_t)})$, and finally define
 \[ M=\begin{pmatrix}
k-\beta&&b_1&&b_2&&\cdots&b_{t-1}&\\
&k-\gamma&&c_1&&c_2&\cdots&&c_{t-1}\\
b_1&&\beta\\
&c_1&&\gamma\\
b_2&&&&\mu_2\\
&c_2&&&&\mu_2\\
&&&&&&\ddots\\
b_{t-1}&&&&&&&\mu_{t-1}\\
&c_{t-1}&&&&&&&\mu_{t-1}
\end{pmatrix}\oplus D_0\]
where the blank entries denotes $0$s.
%
%
Since $M$ is permutationally similar to the block diagonal matrix with blocks $C_\beta$, $C_\gamma$, and $D_0$, the eigenvalues of $M$ 
are $\lambda_1^{(m_1+2)},\ldots, \lambda_t^{(m_t+2)}$.

By the assumption, we may choose $V$ to be a matrix which diagonalizes the matrix $A$ such that for any row $u$, there is a column $j$ corresponding to an eigenvector of some $\mu_\ell$ such that $V_{uj} \not=0$.  Without loss of generality assume that $$V^TAV = \mathrm{diag}(\beta, \gamma, \mu_2^{(2)},\ldots, \mu_{t-1}^{(2)}, \lambda_1^{(m_1)},\ldots, \lambda_t^{(m_t)})=D_2\oplus D_0.$$ Define $W' = I_2 \oplus W_2 \oplus \cdots \oplus W_{t-1}$ where the $W_i$ are any orthogonal $2\times 2$ matrices, and define $W = W' \oplus I_{m_1+\cdots + m_t}$. Then, as $W'$ commutes with $D_2$, we have that $$W^T V^T A V W = V^T A V=D_2\oplus D_0.$$

Let $V'$ be the first $2t-2$ columns of $V$, so that it has columns that are the eigenvectors corresponding to the eigenvalues of $D_2$. Notate these columns by $v_1, \ldots, v_{2t-2}$. Let $U$ be any orthogonal $2\times 2$ matrix. Then 


\begin{align*}
(U \oplus VW) M (U^T \oplus W^TV^T)&=(U \oplus VW) \begin{pmatrix} D_1 & B_0 \\ B_0^T & D_2\oplus D_0 \end{pmatrix} (U^T \oplus W^TV^T)\\
&= \begin{pmatrix} U D_1 U^T & U B_0 W^TV^T\\
VWB_0^T U^T & A\end{pmatrix}.
\end{align*}
This matrix is in $S(K_2\vee G(A))$ if $UD_1 U^T$ has nonzero off-diagonal entries and the following matrix has no zero entry:
\[
UB_0W^TV^T = U B_0'(W')^T(V')^T, \quad\text{where}\quad B_0'=\begin{pmatrix} b_1 & & b_2 & & \cdots & b_{t-1} & \\
& c_1 & & c_2 & \cdots & & c_{t-1}\end{pmatrix}.
\]

Let $\theta_2,\ldots, \theta_{t-1}$ be uniformly and independently chosen angles and let $W_i$ be the $2\times 2$ rotation matrix by angle $\theta_i$. Then the $ij$'th entry of $B_0'(W')^T(V')^T$ is
\[
b_1v_1(j) + \sum_{k=2}^{t-1} b_k v_{2k-1}(j) \cos \theta_k - b_k v_{2k}(j) \sin \theta_k
\]
if $i=1$ and 
\[
c_1v_2(j) + \sum_{k=2}^{t-1} c_k v_{2k-1}(j) \sin \theta_k +c_k v_{2k}(j) \cos \theta_k
\]
if $i=2$. Since the $b_i$ and $c_i$ are nonzero, and by the choice of $V$ there is at least one $u$ with $3\leq u \leq 2t-2$ with $v_u(j) \not=0$, we have that the $ij$'th entry of $B_0'(W')^T(V')^T$ is nonzero with probability $1$. So we can choose $W'$ for which $B_0'(W')^T(V')^T$ has no zero entries. Moreover, since $\beta \not=\gamma$, $D_1$ is not a zero matrix. It is now easy to choose $U$ such that $UD_1U^T$ and $UB_0W^TV^T$ have all nonzero entries.
%
\end{proof}

In this paper we continued the study of the behaviour of $q(G\vee H)$. For a general graph $H$, we obtained 
results for the case when $G$ is either a path or a complete graph, and we explored the potential impact of eigenvector patterns on $q(G\vee H)$, for various families of graphs $H$.

\subsection*{Acknowledgements}
This project started and was made possible by the online research community \textit{Inverse eigenvalue problems for graphs}, which is sponsored by the American Institute of Mathematics with support from the US National Science Foundation. The authors thank AIM and the research community organizers for their support.

We are grateful to the anonymous referee for careful reading and comments, which improved the presentation of the paper.

\bibliographystyle{plain}
\bibliography{references}

\end{document}